\documentclass[preprint,12pt]{elsarticle}
\usepackage{amsthm,amsfonts,amssymb,amscd,amsmath,enumerate,verbatim,calc,graphicx,geometry}
\usepackage[all]{xy}
\newtheorem{theorem}{Theorem}[section]
\newtheorem{lemma}[theorem]{Lemma}

\theoremstyle{definition}
\theoremstyle{definitions}

\theoremstyle{notations}

\theoremstyle{remarks}

\journal{Topology and its Applications}

\begin{document}

\begin{frontmatter}



\title{Burnside Condition on Some  Intersection Subgroups}

 %\author[label1,label2]{<author name>}
% \address[label1]{<address>}
% \address[label2]{<address>}

\author[]{Hanieh~Mirebrahimi\corref{cor1}}
\ead{h\_mirebrahimi@um.ac.ir}
\author[]{Fateme~Ghanei}
\ead{fatemeh.ghanei91@gmail.com}

\address{Department of Pure Mathematics, Center of Excellence in Analysis on Algebraic Structures, Ferdowsi University of Mashhad,\\
P.O.Box 1159-91775, Mashhad, Iran.}
\cortext[cor1]{Corresponding author}

\begin{abstract}
In this paper, using the notions  graphs, core graphs, immersions and covering maps of graphs, introduced by Stallings in 1983, we prove the Burnside condition for the intersection of subgroups of free groups with Burnside condition.
%the $k$th topological coarse shape homotopy group $\check{\pi}_k^{*^{top}}$  can be considered as a quotient space.  %Moreover, we present some basic properties of topological shape homotopy groups.
\end{abstract}


\end{frontmatter}

%\\\\\\\\\\\\\\\\\\\\\\\\\\\\\\\\\\\\\\\\\\\\\\\\\\\\\\\\\\\\\\\\\\\\\\\\\\\\\\\\\\\\\\\\\\\\\\\\\\\\\\\\\\\\\\\\\\\\\\\\\\\\\\\\\\\\\\\\\
%=========================================================================================================================================
%/////////////////////////////////////////////////////////////////////////////////////////////////////////////////////////////////////////
\section{Introduction and Motivation}
 In \cite{St} J. Stallings  studied on  free  groups by theory of graphs. He introduced the concept of immersions of graphs,  provided an algorithmic  process  to study on finitely generated subgroups of free groups. Using these tools, he also gave an elegant proof for Howson's theorem  "if $A$ and $B$ are finitely generated subgroups of a free group, then $A\cap B$ is finitely generated". Moreover,  using immersions of graphs and   core graphs (graphs  with  no  trees  hanging  on)  some mathematicians such as Everitt  and  Gersten  studied on    H. Neumann's  inequality on the rank of $A\cap B$ (see \cite{E}  and \cite{G}). 
 Stallings also in \cite{St}, introduced another notation called "Burnside condition for a subgroup". 
 %Suppose $S$ is a subgroup of a group $G$, we call    $S$ satisfies the Burnside condition
%if  for every $g \in G$, there exists some positive integer $n$ such that $g^{n}\in  S$. 
In this paper, we focus on this notion, and using similar methods,  we  prove that 
if  $A$  and  $B$  are  finitely  generated  subgroups  of 
a  free  group $F$, and    $A \cap  B$  satisfies the Burnside condition  in  both  $A$  and  $B$,  then  $A \cap  B$  satisfies the Burnside condition  in  $A \vee B$,  the  subgroup of $F$ generated  by  $A \cup B$. 
  %//////////////////////////////////////////////////////////////////////////////////////////////////////////////////////////////////////////////
\section{Preliminaries}
In this section, all our notations come from \cite{St}.
%\cite{E1}, \cite{E2} and \cite{E3}.
 A graph $X$ consists of two sets $E$ and $V$ (\textit{edges} and \textit{vertices}), with three functions $^{-1}:E\longrightarrow E$ and $s,t:E\longrightarrow V$ such that $(e^{-1})^{-1}=e$, $e^{-1}\neq e$, $s(e^{-1})=t(e)$ and $t(e^{-1})=s(e)$. We say that the edge $e\in E$ has \textit{initial vertex} $s(e)$ and \textit{terminal vertex} $t(e)$. The edge $e^{-1}$ is also called the \textit{reverse} of $e$.  

A \textit{map of graphs} $f:X\longrightarrow Y$ is a function which maps edges to edges and vertices to vertices. Also we have $f(e^{-1})=f(e)^{-1}$, $f(s(e))=s(f(e))$ and $f(t(e))=t(f(e))$.

A \textit{path} $p$ in $X$ of length $n=\vert p\vert$, with initial vertex $u$ and terminal vertex $v$, is an $n$-tuple of edges of $X$ of the form $p=e_{1}...e_{n}$ such that for $i=1,...,n-1$, we have $t(e_{i})=s(e_{i+1})$ and $s(e_{1})=u$ and $t(e_{n})=v$. For $n=0$, given any vertex $v$, there is a unique path $\Lambda_{v}$ of length $0$ whose initial and terminal vertices coincide and are equal to $v$. A path $p$ is called a \textit{circuit} if its initial and terminal vertices coincide.

If $p$ and $q$ are paths in $X$ and the terminal vertex of $p$ equals the initial vertex of $q$, they may be \textit{concatenated} to form a path $pq$ with $\vert pq\vert=\vert p\vert+\vert q\vert$, whose initial vertex is that of $p$ and whose terminal vertex is that of $q$.

A \textit{round-trip} is a path of the form $ee^{-1}$. A \textit{reduced path} is a path in $X$ containing no round-trip. An \textit{elementary reduction} is insertion or deletion a round-trip in a path. Two paths $p$ and $q$ are \textit{homotopic} (written $p\sim q$) iff there is a finite sequence of elementary reductions taking one path to the other. Homotopic paths must have the same start and terminal vertices and also, homotopy is an equivalence relation on the set of paths with same start and same terminal vertices in $X$. Moreover, any path in $X$ is homotopic to a unique reduced path in $X$.

Let $v$ be a fix vertex in $X$, $\pi_{1}(X,v)$ is  defined to be the set of all homotopy classes of closed paths with initial and terminal vertex $v$. Then $\pi_{1}(X,v)$ together with the product $[p][q]:=[pq]$ forms a group with identity $[\Lambda_{v}]$ and inverse element $[p]^{-1}=[p^{-1}]$. 

For a fix vertex $v$ in $X$, the \textit{star} of $v$ in $X$ is defined as follows:
\begin{center}
$St(v,X)=\lbrace e\in E : s(e)=v\rbrace$.
\end{center}

A map $f:X\longrightarrow Y$ yields, for each vertex $v\in X$, a function $f_{v}:St(v,X)\longrightarrow St(f(v),Y)$. If for each vertex $v\in X$, $f_{v}$ is injective, we call $f$ an \textit{immersion} and if for each vertex $v\in X$, $f_{v}$ is surjective, we call $f$ a \textit{locally surjective}. If each $f_{v}$ is bijective, we call $f$ a \textit{covering map} . 

The theory of coverings of graphs is almost completely analogous to the topological theory of coverings. Immersions have some of the properties of coverings, one of them  we  need it more, is the following one: 

"For a given finite set of elements $\lbrace \alpha_{1},...,\alpha_{n}\rbrace\subseteq \pi_{1}(X,u)$, there is a connected graph $Y$ and an immersion $f:Y\longrightarrow X$ such that $f_{*}(\pi_{1}(Y))=S$, in which $S$ is the subgroup of $\pi_{1}(X,u)$ generated by $\lbrace\alpha_{1},...,\alpha_{n}\rbrace$".

If $G$ is a group, a $G$-\textit{graph} $X$ is a graph with an action of $G$ on the left on $X$ by maps of graphs, such that for all $g\in G$ and every edge $e$, $ge\neq e^{-1}$. In this case, the quotient graph $X/G$, and the natural quotient map of graphs $q:X\rightarrow X/G$ are defined. It is easy to see that, in general $q:X\rightarrow X/G$ is locally surjective. 

We call $G$ \textit{acts freely} on $X$, whenever $v$ is a vertex of $X$, $g\in G$, and $gv=v$, then $g=1$, the identity element of $G$. In this case, $q:X\rightarrow X/G$ is an immersion, and hence is a covering map.

A \textit{translation} of a map of graphs $f:X\rightarrow Y$ is a map $g:X\rightarrow X$ which is an isomorphism of graphs and for which $fg=f$. The set of all translations of $f$ forms a group $G(f)$ which acts on $X$. If $f$ is an immersion, and $X$ is connected, then $G(f)$ acts freely on $X$.

For a connected graph $X$, the \textit{universal cover} $f:\tilde{X}\rightarrow X$ is a covering map with $\tilde{X}$ is connected and $\pi_{1}(\tilde{X})$ is trivial.

\begin{lemma}\label{le}\cite{St}
 If $f:\tilde{X}\rightarrow X$ is a universal covering map, then $G(f)\cong\pi_{1}(X)$ which acts freely, by covering translations, on $\tilde{X}$, and also $f$ is isomorphic to the quotient map $q:\tilde{X}\rightarrow \tilde{X}/G$, i.e., there exists an isomorphism $\varphi:\tilde{X}/G\rightarrow X$ such that $\varphi q=f$.
\end{lemma}
\begin{theorem}\label{cap}\cite{St}
Let 
\begin{equation*}
\label{dia}\begin{CD}
Z_{3}@>g_{1}>>Z_{1}\\
@VV g_{2}V@V f_{1}VV\\
Z_{2}@>f_{2}>>X.
\end{CD}\end{equation*}\\
%\begin{figure}[ht]
%\centerline{\includegraphics[width=2.5cm]{26}}
%\end{figure}\\
be a pullback diagram of graphs, where $f_{1}$ and $f_{2}$ are immersions. Let $v_1$ and $v_2$ be vertices in $Z_1$ and $Z_2$ that $f_1(v_1)=f_2(v_2)=w$. Let $v_3$ be corresponding vertex in $Z_3$. Define $f_{3}=f_{1}g_{1}=f_{2}g_{2}:Z_3\rightarrow X$, and $S_{i}=f_{i_{*}}(\pi_{1}(Z_{i},v_{i}))$, for $i=1,2,3$. Then $S_{3}=S_{1}\cap S_{2}$.
\end{theorem}
%\begin{theorem}\cite{St}
%Let $f:X\longrightarrow Y$ be an immersion of graphs. Suppose that $Y$ has only one vertex and $X$ has only finitely many vertices. Then there exists a graph $X^{\prime}$ containing $X$ such that $X^{\prime}-X$ consists only of edges, and there exists a map $f^{\prime}:X^{\prime}\longrightarrow Y$ extending $f$ such that $f^{\prime}$ is a covering.
%\end{theorem}
%%%%%%%%%%%%%%%%%%%%%%%%%%%%%%%%%%%%%%%%
%%%%%%%%%%%%%%%%%%%%%%%%%%%%%%%%%%%%%%%%%%%%%%%
\section{Main results}
In this section, we deduce our main result. Before it, we recall some notes from \cite{St} which are essential in our proof. First, we note to the \textit{core graphs} whose roles are more important. 

A \textit{cyclically reduced circuit} in a graph $X$ is a circuit $p=e_{1}...e_{n}$, which is reduced as a path and for which $e_{1}\neq e_{n}^{-1}$. A graph $X$ is said to be a \textit{core-graph} if $X$ is connected, has at least one edge and each of its edges belongs to at least one cyclically reduced circuit. 

In   a connected graph $X$ with  non-trivial fundamental group, an \textit{essential edge} is an edge belonging to some cyclically reduced circuit. The \textit{core} of $X$ consists of all essential edges of $X$ and all initial vertices of essential edges. 

If $X$ is a connected graph with non-trivial fundamental group, then the core $X^{\prime}$ of  $X$ is a  core-graph. If $v$ is a vertex of $X^{\prime}$, then the homomorphism induced by  inclusion, $\pi_{1}(X^{\prime},v)\longrightarrow \pi_{1}(X,v)$ is an isomorphism.

Another notion, we are dealing with, is the \textit{Burnside condition} for subgroups. A subgroup $S$  of a group $G$  satisfies the \textit{Burnside condition} if for every $g\in G$, there exists some positive integer $n$ such that $g^{n}\in S$. 

\begin{lemma}\label{im}\cite{St}
$(a)$ Let $f:X\longrightarrow Y$ be a finite-sheeted covering of connected graphs, and $v$ a vertex of $X$. Then $f_{*}(\pi_{1}(X,v))\subseteq \pi_{1}(Y,f(v))$ satisfies the Burnside condition.

$(b)$ Let $f:X\longrightarrow Y$ be an immersion of connected graphs,  $Y$ be a core-graph, $v$ a vertex of $X$, and $f_{*}(\pi_{1}(X,v))\subseteq \pi_{1}(Y,f(v))$ satisfy the Burnside condition. Then $f$ is a covering map.
\end{lemma}

Finally, using all the above notes, we establish our main result in the follow.

\begin{theorem}
Let $S_1$ and $S_2$ be finitely generated subgroups of a free group $F$. If $S_1\cap S_2$ satisfies the Burnside condition both in $S_1$ and $S_2$, then $S_1\cap S_2$  also satisfies the Burnside condition in  $S_1\vee S_2$, the subgroup generated by $S_{1}\cup S_{2}$. 
\end{theorem}
\begin{proof}
Similar to the argument of 7.8 in \cite{St}, we start with the following pullback diagram of immersions
\begin{equation*}
\label{dia}\begin{CD}
Z_{3}@>g_{1}>>Z_{1}\\
@VV g_{2}V@V f_{1}VV\\
Z_{2}@>f_{2}>>X.
\end{CD}\end{equation*}\\
%\begin{figure}[ht]
%\centerline{\includegraphics[width=2.5cm]{26}}
%\end{figure}\\
where $f_{1}$ and $f_{2}$ are immersions, $Z_{3}$ is a subgraph of the pullback; $Z_{1}$, $Z_{2}$ and $Z_{3}$ are core-graphs; $f_{3}=f_{1}g_{1}=f_{2}g_{2}$; $v_{3}$ is a vertex of $Z_{3}$, $v_{1}$ and $v_{2}$ are the images of $v_{3}$ in $Z_{1}$ and $Z_{2}$; $w=f_{i}(v_{i})$ $(i=1,2,3)$, $F=\pi_{1}(X,w)$, $S_{i}=f_{i_{*}}(\pi_{1}(Z_{i},v_{i}))$  $(i=1,2,3)$. Using Theorem \ref{cap}, $S_{3}=S_{1}\cap S_{2}$.

By Lemma \ref{im} (b), since $S_{3}$ satisfies the Burnside condition both in  $S_{1}$ and $S_{2}$, it follows that the immersions $g_{1}$ and $g_2$ are coverings. 

Let $r:\tilde{Z}_{3}\rightarrow Z_{3}$ be a universal covering, and consider $\tilde{g}_{1}=g_{1}r$ and $\tilde{g}_{2}=g_{2}r$, which are consequently universal coveings. Then by Lemma \ref{le}, $G(\tilde{g}_{1})=S_{1}$, $G(\tilde{g}_{2})=S_{2}$ and the universal coverings $\tilde{g}_{1}:\tilde{Z}_{3} \rightarrow  Z_{1}$ and $\tilde{g}_{2}:\tilde{Z}_{3} \rightarrow  Z_{2}$  are isomorphic to quotient maps $q_{1}:\tilde{Z}_{3} \rightarrow  \tilde{Z}_{3}/G(\tilde{g}_{1})$ and $q_{2}:\tilde{Z}_{3} \rightarrow  \tilde{Z}_{3}/G(\tilde{g}_{2})$, respectively, i.e., there are isomorphisms $\varphi_{1}:\tilde{Z}_{3}/G(\tilde{g}_{1})\rightarrow Z_{1}$ and $\varphi_{2}:\tilde{Z}_{3}/G(\tilde{g}_{2})\rightarrow Z_{2}$ such that $\varphi_{1}q_{1}=\tilde{g}_{1}$ and $\varphi_{2}q_{2}=\tilde{g}_{2}$.  By definition, for any covering transformation $\sigma\in G(\tilde{g}_{1})=S_{1}$, $\tilde{g}_{1}\sigma=\tilde{g}_{1}$, and so $g_{1}r\sigma=g_{1}r$. Hence, $f_{3}r\sigma=f_{3}r=h$. So $\sigma$ is a transformation of the immersion $h$. Similarly, any covering transformation of $\tilde{g}_{2}$ is a transformation of $h$. Now, suppose $K$ is the group of transformations of $h$ generated by $G(\tilde{g}_{1})\cup G(\tilde{g}_{2})$, then $K\cong S_{1}\vee S_{2}$ and  we have the following pushout diagram.
%\begin{figure}[ht]
%\centerline{\includegraphics[width=3cm]{2}}
%\end{figure}
\begin{displaymath}
\xymatrix{
\tilde{Z}_{3} \ar[r]^ {q_{1}} \ar[d]_{q_{2}} \ar[dr]^{t_{3}}&
            \tilde{Z}_{3}/G(\tilde{g}_{1})\ar[d]^{t_{1}}\\
\tilde{Z}_{3}/G(\tilde{g}_{2}) \ar[r]_{t_{2}}         & \tilde{Z}_{3}/K }           
\end{displaymath}

Therefore, because of the universal property of pushout,  for graph $X$ and maps $f_{1}\varphi_{1}$ and $f_{2}\varphi_{2}$ there exists a unique map $s:\tilde{Z}_{3}/K \rightarrow X$ such that $st_{1}=f_{1}\varphi_{1}$ and $st_{2}=f_{2}\varphi_{2}$. 

Since $h$ is an immersion, then $G(h)$ and hence  $K$ acts freely on $\tilde{Z}_{3}$, and so $t_{3}$ is a covering. 
%Similarly, one can follows that $q_{1}$ and $q_{2}$ are covering maps.  
It follows that $t_1$ is a covering too. $\tilde{Z}_{3}/G(\tilde{g}_{1})$ is a finite graph and so  $t_{1}$ is a finite-sheeted covering. Thus by Lemma \ref{im} (a), ${t_{1}}_{*}(\pi_{1}(\tilde{Z}_{3}/G(\tilde{g}_{1})))$ satisfies the Burnside condition in $\pi_{1}(\tilde{Z}_{3}/K)$; so $s_{*}(t_{1_{*}}(\pi_{1}(\tilde{Z}_{3}/G(\tilde{g}_{1}))))={f_{1}}_{*}({\varphi_{1}}_{*}(\pi_{1}(\tilde{Z}_{3}/G(\tilde{g}_{1}))))={f_{1}}_{*}(\pi_{1}(Z_{1}))=S_{1}$ satisfies the Burnside condition in $s_{*}(\pi_{1}(\tilde{Z}_{3}/K))$.
%where $s:\tilde{Z}_{3}/K\rightarrow X$ is the unique map exists . 
As we can see, $s_{*}(\pi_{1}(\tilde{Z}_{3}/K))$ contains $S_{1}$ and similarly $S_{2}$ and hence it contains $S_{1}\vee S_{2}$. 
Therefore, since $S_{1}\subseteq S_{1}\vee S_{2}\subseteq s_{*}(\pi_{1}(\tilde{Z}_{3}/K))$ and $S_{1}$ satisfies the Burnside condition in $s_{*}(\pi_{1}(\tilde{Z}_{3}/K))$, then  $S_{1}$ also satisfies the Burnside condition in $S_{1}\vee S_{2}$. 
Now,  by the fact that $S_1\cap S_2$ satisfies the Burnside condition in $S_1$, the result holds. 
\end{proof}
%-------------------------------------------------------------------------------------------------------------------------------------------
%\subsection*{Acknowledgements}
%The authors thank the referee for his/her careful reading and useful suggestions.\\
%This research was supported by a grant from Ferdowsi University of Mashhad; (No. MP89152MSH).
%=========================================================================================================================================
%=========================================================================================================================================



%% References


%\bibliographystyle{elsarticle-num}



\section*{References}

\bibliography{mybibfile}

\end{document}